\documentclass[10pt]{article}

\usepackage[english]{babel}
\usepackage[utf8]{inputenc}
\usepackage[T1]{fontenc}

\usepackage{amsmath, amsthm, amssymb, mathtools}

\usepackage[shortlabels]{enumitem}

\usepackage{cases}
\usepackage{mathtools,tikz-cd}

\usepackage[a4paper]{geometry}

\usepackage[displaymath, mathlines]{lineno}

\AtBeginDocument{\def\MR#1{}}

\usepackage{hyperref}

\providecommand{\keywords}[1]{\textbf{Keywords.} #1}
\providecommand{\MSC}[1]{\textbf{AMS Subject Classifications.} #1}

\newtheorem{theorem}{Theorem}[section]
\newtheorem{corollary}[theorem]{Corollary}
\newtheorem{lemma}[theorem]{Lemma}
\newtheorem{proposition}[theorem]{Proposition}
\newtheorem{example}[theorem]{Example}

\theoremstyle{definition}
\newtheorem{definition}[theorem]{Definition}
\newtheorem{remark}[theorem]{Remark}

\renewcommand\epsilon{\varepsilon}

\usepackage{color}

\usepackage{xspace}

\newcommand{\C}{\field{C}\xspace}
\newcommand{\R}{\field{R}\xspace}

\newcommand{\N}{\field{N}\xspace}

\newcommand{\Herm}{\field{H}\xspace}
\newcommand{\Sym}{\field{S}\xspace}

\newcommand\DDC{D^2_{\C}}
\newcommand\DDR{D^2_{\R}}

\newcommand{\field}[1]{\ensuremath{\mathbb{#1}}}

\newcommand{\ens}[1]{ \left\{#1\right\} }

\newcommand\card{\mathrm{card} \,}

\newcommand\loc{{\mathrm{loc}}}

\def\ds{\displaystyle}

\newcommand{\lin}[1]{\mathcal{L}(#1)}

\newcommand\abs[1]{\left|#1\right|}
\def\norm#1{\left\|#1\right\|}

\newcommand\Id{\mathrm{Id}}

\newcommand\clos[1]{\overline{#1}}

\newcommand\conj[1]{\overline{#1}}

\newcommand\Trace{\mathrm{Trace} \,}
\newcommand\MA{\mathrm{MA}}

\newcommand\real{\mathfrak{Re} \,}

\newcommand{\ps}[2]{ {\left\langle #1 , #2 \right\rangle} }

\newcommand\dist{\mathrm{dist} \,}

\newcommand\weakto{\rightharpoonup}

\newcommand\matcone{\mathcal{C}}
\newcommand\veccone{\Gamma}

\newcommand\matconeR{\mathcal{E}}

\newcommand\degr{d}

\newcommand\st{\quad \middle| \quad}

\newcommand\tra{\mathrm{Tr}}

\title{Local regularity for concave homogeneous complex degenerate elliptic equations comparable to the Monge-Ampère equation}

\AtBeginDocument{\def\MR#1{}}

\author{
Soufian Abja\thanks{Faculty of Mathematics and Computer Science, Jagiellonian University, ul. {\L}ojasiewicza 6, 30-348 Krak\'{o}w, Poland (\texttt{Soufian.Abja@im.uj.edu.pl}, \texttt{math.golive@gmail.com}).}
\and
Guillaume Olive\footnotemark[1]
}

\date{\today}

\begin{document}

\maketitle

\begin{abstract}
In this paper, we establish a local regularity result for $W^{2,p}_{\loc}$ solutions to complex degenerate nonlinear elliptic equations $F(\DDC u)=f$ when they are comparable to the Monge-Ampère equation.
Notably, we apply our result to the so-called $k$-Monge-Ampère equation.
\end{abstract}

\keywords{degenerate elliptic equations, local regularity, Hessian equations}

\vspace{0.2cm}
\MSC{35J70, 32W20}


\section{Introduction and main results}

In this work, we are interested in the local regularity theory for nonlinear complex degenerate elliptic equations of the form
\begin{equation}\label{FDu equ}
F(\DDC u)=f,
\end{equation}
where $\DDC u$ denotes the complex Hessian of $u$.

Such equations have been studied extensively in the literature, going back to the celebrated work \cite{CNS85} on the Dirichlet problem in the real case and its counterpart \cite{Li04} for the complex setting.
We recall that an important feature of degenerate equations is that, unlike uniformly elliptic equations, the $C^{2,\alpha}$-regularity may fail.
Perhaps one of the most important equation of the form \eqref{FDu equ} is the complex Monge-Amp\`ere equation:
\begin{equation}\label{MA equ}
\det (\DDC u)=f>0.
\end{equation}

The regularity of the solution of this equation is studied in the literature by many authors and by different tools.
It was first proved in pioneering work \cite{BT76} that the solution $u$ to the Dirichlet problem in a ball $B$ belongs to $C^{1,1}(B) \cap C^0(\clos{B})$ provided that the right-hand side $f \in C^2(\clos{B})$ and the boundary data is $C^2(\partial B)$.
Another important regularity result concerning the Dirichlet problem associated with \eqref{MA equ} in a strictly pseudoconvex domain, established in \cite{CKNS85}, is the smooth regularity up to the boundary of its solution when the right-hand side, the boundary data and the domain are all smooth (the strict positivity of the right-hand side is also essential).

The local regularity of \eqref{MA equ} (no boundary data) was also studied.
A sharp result was obtained in \cite{BD11} by developing some methods in \cite{Tru80}: solutions $u \in W^{2,p}_{\loc}$ of \eqref{MA equ} with $f \in C^{\infty}$ are necessarily $C^{\infty}$ whenever $p>n(n-1)$, and no smaller exponent $p$ can be expected in general.

The local regularity of other equations of the form \eqref{FDu equ} were also studied.
Notably, in \cite{DK14}, a counterpart of \cite{BD11} was proved for the so-called complex $k$-Hessian equation under the assumption that $u$ belongs to $W^{2,p}_{\loc}$ with $p> n(k-1)$.
 
The goal of the present paper is to extend the approach of \cite{BD11} and \cite{DK14} to more general nonlinear complex degenerate elliptic equations.
We will introduce simple conditions on the nonlinearity $F$ to obtain general local regularity results, thus considerably broadening the field of application of this method.

In particular, we shall see that our results apply to the so-called complex $k$-Monge-Amp\`ere equation or $\MA_k$-equation, $k \in \ens{1,\ldots,n}$, that has recently received much attention (\cite{HL09,Sad18,Din20-pre-1, Din20-pre-2}):
$$
\prod_{1 \leq i_1<\ldots<i_k \leq n} \left(\lambda_{i_1}(\DDC u)+\ldots+\lambda_{i_k}(\DDC u)\right)=f,$$
where $\lambda_1(\DDC u), \ldots, \lambda_n(\DDC u)$ denote the eigenvalues of $\DDC u$.
For $k=1$ we have the Monge-Ampère equation $\det (\DDC u)=f$ and for $k=n$ this is the Poisson equation $\Delta u=f$.

Interior estimates for this equation in the real setting have been studied recently in \cite{Din20-pre-2}.
In the complex setting this operator was discussed in \cite{Sad18} and, in the special case $k=n-1$, also in \cite{TW17}.
It has been shown in \cite{Din20-pre-1} that this operator does not satisfy an integral comparison principle, which makes the associated potential theory much harder to be developed.
Finally, let us also mention that the Dirichlet problem associated to this operator was studied in \cite{Zho13} using a probabilistic approach.

The outline of the paper is as follows.
After recalling some basic notations, we precise in Section \ref{sect main thm} what kind of nonlinear operators $F$ are considered in this article and we present our main results.
Their proofs are the purpose of Section \ref{sect proof main thm}.
Some examples of Hessian equations covered by such a framework are then given in Section \ref{sect examples}.
Finally, in Section \ref{sect Mk} we detail the case study of the $\MA_k$-equation.

\subsection{General notations}\label{sect nota}

\begin{itemize}
\item
All along this work, $\Omega \subset \C^n$ ($n \geq 1$) is a nonempty open bounded connected subset.

\item
Let $\Herm^n$ be the set of $n \times n$ Hermitian matrices.
The $(i,j)$-th entry of a matrix $A \in \Herm^n$ will be denoted by $a_{i\bar{j}}$.
We recall that $\Herm^n$ is a vector space over the field $\R$ of dimension $n^2$. The inner product on $\Herm^n$ is the standard Frobenius inner product:
$$\ps{A}{B}_{\Herm^n}=\Trace(AB) \in \R, \quad A,B \in \Herm^n.$$

The $n \times n$ identity matrix will be denoted by $\Id$.

We will use the classical cones
$$
\matcone_n=\ens{A \in \Herm^n \st A>0},
\quad
\matcone_1=\ens{A \in \Herm^n \st \Trace(A)>0}.
$$

\item
The Fréchet derivative at $A \in \matcone$ of a function $F \in C^1(\matcone,\R)$, defined on a non empty open subset $\matcone \subset \Herm^n$, will be denoted by $DF(A) \in \lin{\Herm^n,\R}$.
It can be identified with the matrix $\left(\frac{\partial F}{\partial a_{i\bar{j}}}(A)\right)_{1 \leq i,j \leq n} \in \Herm^n$ through the formula
$$DF(A) B=\sum_{1 \leq i,j \leq n} \frac{\partial F}{\partial a_{i\bar{j}}}(A) b_{i\bar{j}}, \quad B \in \Herm^n.$$

\item
In this work we use the following convention for ellipticity.
We say that a function $F:\matcone \longrightarrow \R$ is:
\begin{itemize}
\item
(degenerate) elliptic in $\matcone$ if
$$A \leq B \quad \Longrightarrow \quad F(A) \leq F(B),$$
for every $A,B \in \matcone$, where we recall that $A \leq B$ (resp. $A<B$) means that the Hermitian matrix $B-A$ is non-negative-definite (resp. positive-definite).

\item
uniformly elliptic in $\matcone$ if there exist $0<m \leq M$ such that
$$A \leq B \quad \Longrightarrow \quad  m\norm{B-A}_{\Herm^n} \leq F(B)-F(A) \leq M\norm{B-A}_{\Herm^n},$$
for every $A,B \in \matcone$.
\end{itemize}

For instance with these conventions the equation $\Delta u=f$ is (uniformly) elliptic.

\item
The complex Hessian is denoted by $\DDC u=\left(u_{i \bar{j}}\right)_{1 \leq i,j \leq n}$, where we use the standard notations $u_j$ and $u_{\bar{j}}$ to denote, respectively, $\frac{\partial{u}}{\partial z_j}=\frac{1}{2}\left(\frac{\partial{u}}{\partial x_j}-i\frac{\partial{u}}{\partial y_j}\right)$ and $\frac{\partial{u}}{\partial \bar{z}_j}=\frac{1}{2}\left(\frac{\partial{u}}{\partial x_j}+i\frac{\partial{u}}{\partial y_j}\right)$.
The Laplacian is denoted by $\Delta u=\sum_{j=1}^n u_{j \bar{j}}=\Trace(\DDC u) \in \R$.
Note that it is the complex Laplacian, so it is one fourth the real Laplacian.

\item
Finally, we shall denote by $C(n), C(n,\Omega)$, etc. a positive number that may change from line to line but that depends only on the quantities indicated between the brackets.

\end{itemize}

\subsection{Main results}\label{sect main thm}

Throughout this work we will always make the following assumptions:
\begin{enumerate}[(a)]
\item\label{domain}
\textbf{Domain:}
$\matcone \subset \Herm^n$ is a nonempty open convex cone such that $\matcone_n\subset\matcone \subset \matcone_1$.

\item\label{reg f}
\textbf{Regularity:}
$F \in C^1(\clos{\matcone},\R)$.

\item\label{pos f}
\textbf{Positivity:}
$F>0$ in $\matcone$.

\item\label{F hom}
\textbf{Homogeneity:}
$F$ is positively homogeneous of degree $\degr>0$ (i.e. $F(\alpha A)=\alpha^\degr F(A)$ for ever $\alpha>0$ and $A \in \matcone$).

\item\label{hyp cvx}
\textbf{Concavity:}
$F^{1/\degr}$ is concave in $\matcone$.

\item\label{hyp comp with MA}
\textbf{Comparison with determinant:}
there exists $C>0$ such that
\begin{equation}\label{ineq Garding F}
F(P)^{1/d} \geq C (\det(P))^{1/n}, \quad \forall P \in \matcone_n.
\end{equation}

\end{enumerate}

Some examples will be presented in Section \ref{sect examples} below.

The key assumption here is the assumption \ref{hyp comp with MA}, that will help us to compare our general nonlinear equation $F(\DDC u)=f$ with the complex Monge-Ampère equation and thus use the recent uniform estimates obtained in \cite{ADO20}.

The assumption \ref{domain} is quite standard (see e.g. \cite[pp. 261-262]{CNS85}).
Note that this excludes the case of the whole space $\matcone=\Herm^n$.

It is important to point out that we cannot only consider homogeneous functions of degree $1$ by normalizing $F$ into $F^{1/\degr}$ since this normalization is not in general regular up to the boundary (e.g. $F(A)=\det(A)^{1/n}$ in $\matcone=\matcone_n$).

Finally, by concavity and homogeneity, we can check that $F$ satisfies \eqref{ineq Garding F} for every $A \in \matcone$ if, and only if,
$$F(A+P)^{1/\degr} \geq F(A)^{1/\degr}+C\left(\det P\right)^{1/n},$$
 for every $A \in \matcone$ and $P \in \matcone_n$.
In particular, $F$ is necessarily elliptic in $\matcone$ (but, in general, not uniformly elliptic).
We refer for instance to \cite[Remark 1.3]{ADO20} for more details.

The first result of the present paper is the following:

\begin{theorem}\label{main thm}
Let $\matcone$ and $F$ satisfy the above standing assumptions \ref{domain}, \ref{reg f}, \ref{pos f}, \ref{F hom}, \ref{hyp cvx} and \ref{hyp comp with MA}.
Let $f>0$ in $\Omega$ with $f^{1/\degr} \in C^{1,1}(\Omega)$.
Let $u \in W^{2,p}_{\loc}(\Omega)$ with $(\DDC u)(z) \in \matcone$ for a.e. $z \in \Omega$ satisfy (almost everywhere)
$$F(\DDC u)=f \quad \mbox{ in } \Omega,$$
and assume that
\begin{equation}\label{cond p}
p>n\max\ens{\degr-1,1}.
\end{equation}
Then, for every non empty open subset $\omega \subset \subset \Omega$, we have $\Delta u \in L^{\infty}(\omega)$ with
$$\sup_{\omega} \Delta u \leq R,$$
for some $R>0$ depending only on $n,p,\degr,\dist(\omega,\partial\Omega),\norm{\DDC u}_{L^p(\Omega_\epsilon)}$, $\norm{\Delta (f^{1/\degr})}_{L^{\infty}(\Omega_\epsilon)}$ and $\inf_{\Omega_\epsilon} f$ if $\degr>1$ (resp. $\sup_{\Omega_\epsilon} f$ if $\degr<1$), where $\Omega_\epsilon=\ens{z \in \Omega \st \dist(z,\partial\Omega)>\epsilon}$ and $\epsilon>0$ is some number depending only on $\dist(\omega,\partial\Omega)$.
\end{theorem}

Theorem \ref{main thm} generalizes \cite[Theorem]{BD11} and \cite[Theorem 4.1]{DK14}, where the particular cases of the complex Monge-Ampère equation and the complex $k$-Hessian equation were considered.


An important consequence of Theorem \ref{main thm} and of the classical theory for uniformly elliptic real equations from \cite{CC95} will be the following local regularity result:

\begin{corollary}\label{cor main thm}
Under the framework of Theorem \ref{main thm}, assume in addition that $F$ satisfies the following property:
\begin{enumerate}[(a)]
\setcounter{enumi}{6}
\item\label{prop for UE}
\textbf{Sets of uniform ellipticity:}
there exists $R_0>0$ such that, for every $R>R_0$, $F$ is uniformly elliptic in $\matcone_R$, where
\begin{equation}\label{def CR}
\matcone_R=\ens{A \in \matcone \st  \Trace(A)<R, \quad
\frac{1}{R}<F(A)}.
\end{equation}
\end{enumerate}
Then, we have the property
\begin{equation}\label{Cinf reg}
F \in C^{\infty}(\matcone), \quad f \in C^{\infty}(\Omega) \quad \Longrightarrow \quad u \in C^{\infty}(\Omega).
\end{equation}

\end{corollary}

\begin{remark}
Note that $\matcone_R$ is convex because $F$ is concave.
However, it is never a cone and it does not contain $\matcone_n$.
Note as well that $\lim_{R \to +\infty} \matcone_R=\matcone$.

We also point out that we do not need to assume that $\clos{\matcone_R}$ is compact.
For instance, $\matcone_R$ is not bounded for the simplest example $F(A)=\Trace(A)$ on $A \in \matcone_1$ ($n \geq 2$).
However, this is a rather pathological case and we will see in Proposition \ref{prop hatF} below that $\matcone_R$ is always bounded whenever
$$\partial \matcone \cap \partial \matcone_1=\ens{0}.$$
\end{remark}

\begin{remark}
Whether the condition \eqref{cond p} yields the smallest value of $p$ to induce the $C^{\infty}$-regularity \eqref{Cinf reg} may depend on the equation itself.
This condition is for instance sharp in the case of the Monge-Ampère equation, see \cite[p. 412]{BD11}, but it is still an open problem for other equations such as the complex $k$-Hessian equation, the $k$-Monge-Ampère equation (see Section \ref{sect lower bound} below), etc.
\end{remark}

The rest of the article is organized as follows.
In the next section we prove our main results.
In Section \ref{sect examples} we present some examples of Hessian equations that are covered by our framework.
Finally, in Section \ref{sect Mk}, we discuss a bit more in detail the case of the $k$-Monge-Ampère equation.

\section{Proofs of the main results}\label{sect proof main thm}

The proof of Theorem \ref{main thm} is inspired from the proof of \cite[Theorem 1]{Tru80} and the ideas of \cite{BD11} (see also \cite[Section 4]{DK14}).

We introduce the normalization
$$G(A)=(F(A))^{1/\degr}, \quad A \in \matcone.$$
Clearly, $G$ is elliptic in $\matcone$.
The equation $F(\DDC u)=f$ then becomes
$$G(\DDC u)=f^{1/\degr}.$$

Throughout Section \ref{sect proof main thm} the notation $L_u$ is exclusively saved for the linearization of $G$ about $A=(\DDC u)(z)$:
\begin{equation}\label{def Lu}
(L_u w)(z)=\sum_{1 \leq i,j \leq n} \frac{\partial G}{\partial a_{i\bar{j}}}((\DDC u)(z)) w_{i\bar{j}}(z).
\end{equation}

Note that $L_u$ is (degenerate) elliptic and that its coefficients are Lebesgue measurable (as composition of a continuous function with a Lebesgue measurable function).

\subsection{Preliminaries}

\paragraph{Geometric configuration.}

Let $\omega \subset \subset \Omega$ and set $\delta=\dist(\omega,\partial\Omega)>0$.
For $z_0 \in \omega$, let
$$\tilde{\omega}=\omega \cap B_{\delta/4}(z_0),$$
where $B_R(z_0) \subset \C^n$ denotes the open ball of center $z_0$ and radius $R>0$.
Up to the transformation $z \mapsto (z-z_0)/(\delta/2)$, we can always assume that $B_{\delta/2}(z_0)$ is the open unit ball, that will simply be denoted by $B_1$ in the sequel.
Thus, $B_\delta(z_0)$ becomes $B_2(0)$, that will simply be denoted by $B_2$.
Therefore, from now on, we are in the following geometric configuration:
\begin{equation}\label{geo config}
\tilde{\omega} \subset\subset B_1 \subset\subset B_2 \subset\subset \Omega,
\quad
\dist(\tilde{\omega},\partial B_1) \geq \frac{1}{4}\dist(\omega,\partial\Omega).
\end{equation}

\paragraph{Approximation of the Laplacian.}

In order not to consume too much regularity, we will need a suitable approximation of the Laplacian.
This is done as in \cite[p. 415]{BD11} and \cite[Lemma 4.2]{DK14} (see also \cite[Proposition 6.3]{BT76}).

\begin{lemma}\label{lem prop approxi}
For every $0<\epsilon<1$, let $T_\epsilon:L^p(B_2) \longrightarrow L^p(B_1)$ ($1 \leq p \leq \infty$) be the linear operator defined by
$$(T_\epsilon u)(z)=
\frac{n+1}{\epsilon^2} \left(u_{\epsilon}(z)-u(z)\right),$$
with
$$u_{\epsilon}(z)=\frac{1}{\mu_n(B_\epsilon(z))} \int_{B_\epsilon(z)} u \, d\mu_n,$$
where here and in what follows $\mu_n$ denotes the Lebesgue measure in $\C^n$.
Then, we have the following properties:
\begin{enumerate}[(i)]
\item\label{Tu positive}
\textbf{Positivity:}
If $u$ is subharmonic in $B_2$, then $(T_\epsilon u)(z) \geq 0$ for every $z \in B_1$.

\item\label{reg ueps}
\textbf{Regularity:}
If $u \in C^0(B_2)$, then $T_\epsilon u \in C^0(\clos{B_1})$.
If $u \in W^{2,p}(B_2)$ and $p<\infty$, then $u_\epsilon \in W^{2,p}(B_1)$ with $\DDC u_\epsilon=(\DDC u)_\epsilon$ (with a slight but obvious abuse of notation).

\item\label{item CV}
\textbf{Convergence:}
If $u \in W^{2,p}(B_2)$ and $p<\infty$, then $T_\epsilon u \rightharpoonup \Delta u$ weakly in $L^p(B_1)$ as $\epsilon \to 0$.

\item\label{UB lem}
\textbf{Uniform bound:}
If $u \in W^{2,p}(B_2)$ then, for every $0<\epsilon<1$, we have
$$\norm{T_\epsilon u}_{L^p(B_1)} \leq C(n,p) \norm{\Delta u}_{L^p(B_2)}.$$

\end{enumerate}

\end{lemma}

The first point is a direct consequence of the mean value inequality (see e.g. \cite[Theorem 2.4.1]{Kli91}).
The second point is not difficult to check.
For the weak convergence one can for instance adapt the proof of \cite[Proposition 4.2.6]{Kli91}.
Finally, for the uniform bound, see e.g. \cite[p. 219]{FSX18}.

\paragraph{A uniform estimate for strong supersolutions.}

Recently, it was obtained in \cite[Theorem 1.4]{ADO20} some new type of Alexandrov-Bakelman-Pucci estimate for various nonlinear complex degenerate equations.
Thanks to the assumptions of the present article, it applies in particular to our linearized operator $L_u$ defined in \eqref{def Lu} for which it results in the following:

\begin{theorem}\label{max pp}
Let $r,q>n$.
For every $g \in L^q(B_1)$ and every $w \in W^{2,r}_{\loc}(B_1) \cap C^0(\clos{B_1})$ such that
$$
\begin{dcases}
\begin{aligned}
L_u(-w) & \leq g && \text{ in } B_1, \\
w & \leq 0 && \text{ on } \partial B_1,
\end{aligned}
\end{dcases}
$$
we have
$$
\sup_{B_1} w
\leq C(n,r,q)\norm{g_+}_{L^q(B_1)},
$$
where $g_+=\max(g,0)$ denotes the positive part of $g$.
\end{theorem}

\paragraph{Estimate of the operator norm.}

We will need the following estimate:

\begin{lemma}
There exists $C>0$ such that, for every $A \in \matcone$,
\begin{equation}\label{estim DG}
\norm{DG(A)}_{\lin{\Herm^n,\R}} \leq C (G(A))^{1-\degr} \norm{A}^{\degr-1}_{\Herm^n}.
\end{equation}
\end{lemma}

\begin{proof}
By continuity of $DF$ on the compact $\ens{A \in \clos{\matcone} \st \norm{A}_{\Herm^n}=1}$ and the fact that $DF$ is homogeneous of degree $\degr-1$, we have, for some $C>0$,
$$
\norm{DF(A)}_{\lin{\Herm^n,\R}} \leq C \norm{A}^{\degr-1}_{\Herm^n}, \quad \forall A \in \matcone.
$$
The desired estimate \eqref{estim DG} then follows from the computation
$$DG(A)B=\frac{1}{\degr} (F(A))^{(1/\degr)-1} DF(A)B, \quad A \in \matcone, \, B \in \Herm^n.$$
\end{proof}

We point out that it is only to show \eqref{estim DG} that we use that $F$ is $C^1$ up to the boundary.

Despite not needed, we mention that the inequality \eqref{estim DG} also shows that the coefficients of the operator $L_u$ belong to $L^{\infty}$ if $\degr=1$, and to $L^{p/(\degr-1)}$ if $\degr>1$, which are then in $L^n$ by the assumption \eqref{cond p}.

\paragraph{Jensen's inequality in convex subsets of $\Herm^n$.}

The following version of Jensen's inequality will be needed (see e.g. \cite[Lemma 2.8.1, p.76]{Fer67}):

\begin{lemma}\label{lem Jensen}
Let $\matcone \subset E$ be a non empty open convex set of a real finite dimensional space $E$ and let $G:\matcone \longrightarrow \R$ be a concave function.
Let $\Omega \subset \C^n$ ($n \geq 1$) be a non empty open bounded subset and $H \in L^1(\Omega,E)$ be such that $H(z) \in \matcone$ for a.e. $z \in \Omega$.
Then, we have
$$\frac{1}{\mu_n(\Omega)}\int_\Omega H(z) \, d\mu_n \in \matcone,$$
and
$$
G\left(
\frac{1}{\mu_n(\Omega)}
\int_\Omega H(z) \, d\mu_n
\right)
\geq
\frac{1}{\mu_n(\Omega)}\int_\Omega G\left(H(z)\right) \, d\mu_n,
$$
(the right-hand side possibly being equal to $-\infty$).
\end{lemma}

\subsection{Proof of Theorem \ref{main thm}}

We are now ready to prove our first main result.

\begin{enumerate}[1)]
\item
\textbf{Strategy of the proof.}

Let $0<\epsilon<1$.
For $\alpha,\beta \in [2,+\infty)$ (to be determined later) let us introduce an auxiliary function $w$ given by
$$w(z)=\eta(z) (T_\epsilon u(z))^{\alpha}, \quad z \in B_1,$$
where
$$\eta(z)=\dist(z,\partial B_1)^{2\beta}
=\left(1-\abs{z}^2\right)^\beta.$$

Note that:
\begin{itemize}
\item
$w$ is unambiguously defined since $T_\epsilon u \geq 0$ by item \ref{Tu positive} of Lemma \ref{lem prop approxi} and subharmonicity of $u$, which follows from the assumptions that $\DDC u \in \matcone \subset \matcone_1$ (see e.g. \cite[Theorem 2.5.8]{Kli91}).

\item
$w \in C^0(\clos{B_1})$ by item \ref{reg ueps} of Lemma \ref{lem prop approxi} since $u \in C^0(\Omega)$ by the Sobolev embedding $W^{2,p}_{\loc}(\Omega) \subset C^0(\Omega)$ as $p>n$ by assumption \eqref{cond p}.

\item
$w \in W^{2,p}_{\loc}(B_1)$ thanks to the regularity of $u$ and Sobolev embeddings (using that $p \geq n$).

\end{itemize}

The goal will be to bound $\sup_{B_1} w$ from above by a positive number depending only on the quantities indicated in the statement of Theorem \ref{main thm}.
This will also provide an upper bound for $\sup_{\tilde{\omega}} T_\epsilon u$ since $\eta(z) \geq \left(\dist(\omega,\partial\Omega)/4\right)^{2\beta}>0$ for $z \in \tilde{\omega}$ (recall \eqref{geo config}).
This will in turn imply that $\Delta u \in L^{\infty}(\tilde{\omega})$ with the same bound as for $T_\epsilon u$.
Indeed, denoting by $C$ a bound for $T_\epsilon u$ we will have shown that $T_\epsilon u \in S$, where
$$S=\ens{g \in L^{\infty}(\tilde{\omega}), \quad \norm{g}_{L^{\infty}(\tilde{\omega})} \leq C}.$$
We can check that this set is closed in $L^p(\tilde{\omega})$ (using for instance the partial converse of the Lebesgue dominated convergence theorem \cite[Theorem 3.12]{Rud87}).
Since it is also clearly convex, it is then weakly closed in $L^p(\tilde{\omega})$ (see e.g. \cite[Theorem 3.12]{Rud91}).
As $T_\epsilon u \weakto \Delta u$ weakly in $L^p(\tilde{\omega})$ (item \ref{item CV} of Lemma \ref{lem prop approxi}) and $T_\epsilon u \in S$, it follows that $\Delta u \in S$ as well, which is exactly what we want.

Now, in order to bound $\sup_{B_1} w$ we are going to use the condition \eqref{cond p} on $p$ to show that $(L_u (-w))_+ \in L^q(B_1)$ for some $q>n$, with estimate
\begin{equation}\label{estim g}
\norm{(L_u (-w))_+}_{L^q(B_1)} \leq C\left(\left(\sup_{B_1} w\right)^{1-2/\beta}+1\right),
\end{equation}
for some $C>0$ depending only on the quantities indicated in the statement of Theorem \ref{main thm}.
The conclusion will then follow from the uniform estimate of Theorem \ref{max pp} with $g=(L_u (-w))_+$.
Note that $w=0$ on $\partial B_1$ (in fact, this is the (only) obstruction to simply take $\eta=1$).

\item
\textbf{Computation of $L_u(-w)$.}

We have
$$w_i=\eta_i (T_\epsilon u)^{\alpha}+\eta \alpha (T_\epsilon u)^{\alpha-1} (T_\epsilon u)_i,$$
and
\begin{multline*}
w_{i \bar{j}}=
\eta_{i \bar{j}} (T_\epsilon u)^{\alpha}
+\eta_i \alpha (T_\epsilon u)^{\alpha-1}(T_\epsilon u)_{\bar{j}}
+\eta_{\bar{j}} \alpha (T_\epsilon u)^{\alpha-1} (T_\epsilon u)_i
\\
+\eta \alpha(\alpha-1) (T_\epsilon u)^{\alpha-2} (T_\epsilon u)_{\bar{j}} (T_\epsilon u)_i
+\eta \alpha (T_\epsilon u)^{\alpha-1} (T_\epsilon u)_{i \bar{j}}.
\end{multline*}

Consequently,
$$
L_u (-w)=
\sum_{1\leq i, j \leq n} -\frac{\partial G}{\partial a_{i \bar{j}}}(\DDC u) w_{i \bar{j}}=
I+II,
$$
where (using also the identity $\overline{v_i}=v_{\bar{i}}$ for any real-valued function $v$)
\begin{multline*}
I=
\sum_{1\leq i, j \leq n} -\frac{\partial G}{\partial a_{i \bar{j}}}(\DDC u) \eta_{i \bar{j}} (T_\epsilon u)^{\alpha}
+2 \alpha (T_\epsilon u)^{\alpha-1} \real \left(\sum_{1\leq i, j \leq n} -\frac{\partial G}{\partial a_{i \bar{j}}}(\DDC u)  \eta_i (T_\epsilon u)_{\bar{j}}\right)
\\
+\sum_{1\leq i, j \leq n} -\frac{\partial G}{\partial a_{i \bar{j}}}(\DDC u) \eta \alpha(\alpha-1) (T_\epsilon u)^{\alpha-2} (T_\epsilon u)_{\bar{j}} (T_\epsilon u)_i,
\end{multline*}
and
$$II=\sum_{1\leq i, j \leq n} -\frac{\partial G}{\partial a_{i \bar{j}}}(\DDC u) \eta \alpha (T_\epsilon u)^{\alpha-1} (T_\epsilon u)_{i \bar{j}}.$$

\item
\textbf{Estimate of the term $I$.}

To estimate this term we will make use of the ellipticity and of the estimate \eqref{estim DG}.
Let us introduce the following sesquilinear form on $\C^n$:
$$\varphi(a,b)=\sum_{1\leq i, j \leq n} -\frac{\partial G}{\partial a_{i \bar{j}}}(\DDC u) a_i \conj{b_j}, \quad a,b \in \C^n.$$

This form is nonpositive by ellipticity of $G$.
Therefore, we can use the basic inequality $2 \real \varphi(a,b) \leq -\varphi(a,a) -\varphi(b,b)$ with $a_i=\frac{1}{\sqrt{t}}\eta_i$ and $b_j=\sqrt{t}(T_\epsilon u)_j$ ($t>0$ to be chosen below), to obtain
\begin{multline*}
2\real \left(\sum_{1\leq i, j \leq n} -\frac{\partial G}{\partial a_{i \bar{j}}}(\DDC u)  \eta_i (T_\epsilon u)_{\bar{j}}\right)
\leq 
-\frac{1}{t}\sum_{1\leq i, j \leq n} -\frac{\partial G}{\partial a_{i \bar{j}}}(\DDC u) \eta_i \eta_{\bar{j}}
\\
-
t\sum_{1\leq i, j \leq n} -\frac{\partial G}{\partial a_{i \bar{j}}}(\DDC u) (T_\epsilon u)_i (T_\epsilon u)_{\bar{j}}.
\end{multline*}

Taking now $t=(\alpha-1)\eta/T_\epsilon u>0$, this removes the third term in $I$ to give
$$
I \leq 
(T_\epsilon u)^{\alpha}
\sum_{1\leq i, j \leq n} -\frac{\partial G}{\partial a_{i \bar{j}}}(\DDC u) \left(
\eta_{i \bar{j}}
-\frac{\alpha}{(\alpha-1) \eta} \eta_i \eta_{\bar{j}}
\right).
$$

Let
$$
B=(b_{i\bar{j}})_{1 \leq i,j \leq n},
\quad
b_{i \bar{j}}=
-\eta_{i \bar{j}}
+\frac{\alpha}{(\alpha-1) \eta} \eta_i \eta_{\bar{j}}.
$$
Direct computations show that
$$\eta_i=-\beta \bar{z}_i \eta^{1-1/\beta},
\quad
\eta_{i \bar{j}}=-\beta \delta_{i j} \eta^{1-1/\beta}
+\beta(\beta-1)\bar{z}_i z_j \eta^{1-2/\beta},
$$
where $\delta_{ij}$ denotes the Kronecker delta.
Since $\eta \leq 1$, we obtain
$$
\norm{B}_{\Herm^n} \leq C(\alpha,\beta,n)\eta^{1-2/\beta}
=C(\alpha,\beta,n) w^{1-2/\beta} \frac{1}{(T_\epsilon u)^{\alpha(1-2/\beta)}}.
$$

Consequently, using \eqref{estim DG}, we have
$$
I
\leq
C(\alpha,\beta,n)
C'(f,\degr)
\norm{\DDC u}^{\degr-1}_{\Herm^n} (T_\epsilon u)^{2\alpha/\beta}
w^{1-2/\beta},
$$
where $C'(f,\degr)=(\inf_{B_2} f)^{(1/\degr)-1}$ if $\degr \geq 1$ and $C'(f,\degr)=(\sup_{B_2} f)^{(1/\degr)-1}$ if $\degr<1$.
Note that we cannot use the better estimate \eqref{ineq Garding F} since $B \not\in \clos{\matcone_n}$.

\item
\textbf{Estimate of the term $II$.}

For this term, we will use the concavity of $F$ to show that
$$
II \leq  \eta \alpha \abs{T_{\epsilon} (f^{1/\degr})} (T_\epsilon u)^{\alpha-1}.
$$

To prove such an estimate, we would like to use the concavity inequality
\begin{equation}\label{cvx ineq}
G(B) \leq G(A)+DG(A)(B-A), \quad A,B \in \matcone,
\end{equation}
with
$$A=(\DDC u)(z), \quad B=(\DDC u_{\epsilon})(z).$$

But first observe that $B=(\DDC u)_\epsilon(z)$ (see Lemma \ref{lem prop approxi}) and apply Jensen's inequality (Lemma \ref{lem Jensen}) to see that indeed $B \in \matcone$  for a.e. $z \in B_1$ with, in addition,
$$G((\DDC u)_{\epsilon})
\geq \left(G(\DDC u)\right)_{\epsilon}.$$

We can now use \eqref{cvx ineq} to obtain the desired estimate:
$$
\sum_{1\leq i, j \leq n} -\frac{\partial G}{\partial a_{i \bar{j}}}(\DDC u) (T_\epsilon u)_{i \bar{j}}
\leq \frac{n+1}{\epsilon^2}\left(G(\DDC u)-\left(G(\DDC u)\right)_{\epsilon}\right)=-T_\epsilon (f^{1/\degr}).
$$

\item
\textbf{Choices of $\alpha$ and $\beta$.}

In summary, we have obtain the inequality
$$
(L_u (-w))_+ \leq C(\alpha,\beta,n)C'(f,\degr)
\norm{\DDC u}^{\degr-1}_{\Herm^n} (T_\epsilon u)^{2\alpha/\beta}
w^{1-2/\beta}
+\eta \alpha \abs{T_\epsilon (f^{1/\degr})} (T_\epsilon u)^{\alpha-1}.
$$

Thanks to the condition \eqref{cond p} on $p$, we can find $q>n$ such that $p/q \geq 1$ and $(p/q)-(\degr-1)>0$.
Consequently, taking
$$
\alpha=1+\frac{p}{q},
\quad
\beta=\frac{2 \alpha}{\frac{p}{q}-(\degr-1)},
$$
we have $\alpha,\beta \geq 2$ and $(L_u (-w))_+ \in L^{q}(B_1)$, with the following estimate (by H{\"o}lder's inequality and \ref{UB lem} of Lemma \ref{lem prop approxi}):
\begin{multline*}
\norm{(L_u (-w))_+}_{L^{q}(B_1)}
\leq C(\alpha,\beta,n)C'(f,\degr)C(n,p)
\norm{\DDC u}_{L^p(B_1)}^{\degr-1}
\norm{\Delta u}_{L^p(B_2)}^{2\alpha/\beta}
\sup_{B_1} w^{1-2/\beta}
\\
+\eta \alpha C(n,p) \norm{\Delta (f^{1/\degr})}_{L^{\infty}(B_2)} \norm{\Delta u}_{L^p(B_2)}^{p/q}.
\end{multline*}
This yields the desired estimate \eqref{estim g}.
\qed


\end{enumerate}

\subsection{$C^\infty$ regularity}

In this section, we show how to deduce Corollary \ref{cor main thm} from Theorem \ref{main thm} and the classical theory for uniformly elliptic real equations from \cite{CC95}.

We follow the presentations of \cite{Wan12-MRL,TWWY15}.
Let $\Sym^{2n}$ be the space of $2n \times 2n$ real symmetric matrices.
The inner product on $\Sym^{2n}$ is the standard Frobenius inner product.

In this section, a function $u$ of $n$ complex variables $(z_1,\ldots,z_n)$ will also be a considered as a function of $2n$ real variables $(x_1,\ldots, x_n, y_1,\ldots,x_n)$, where $z_j=x_j+iy_j$.
The real Hessian of $u$ is then denoted by
$$
\renewcommand\arraystretch{2}
\DDR u
=
\left(\begin{array}{c|c}
u_{x_i x_j} & u_{x_i y_j} \\
\hline
u_{y_i x_j} &u_{y_i y_j}
\end{array}\right)
\in \Sym^{2n}.
$$
The open ball in $\R^{2n}$ of center $0$ radius $r>0$ will be denoted by $B_r$.

Corollary \ref{cor main thm} is essentially a consequence of the following fundamental general result and Schauder's estimates:

\begin{theorem}\label{thm CC95}
Let $\bar{F}:\Sym^{2n} \longrightarrow \R$ be concave and uniformly elliptic, and let $f \in C^{0,\beta}(\clos{B_1})$ ($0<\beta \leq 1$).
If $u \in W^{2,p}_{\loc}(B_1)$ is a strong solution to
$$\bar{F}(\DDR u)=f \quad \mbox{ in } B_1,$$
and $p \geq 2n$, then there exist $\alpha \in (0,1)$ and $\epsilon \in (0,1)$ such that
$$u \in C^{2,\alpha}(B_\epsilon).$$
\end{theorem}

We recall that ``strong solution'' simply means that the equation holds almost everywhere.
This theorem is a consequence of \cite[Theorem 8.1]{CC95} (see e.g. \cite[Section 2]{Wan12-MRL} for details).
In this reference the result deals in fact with $C$-viscosity solutions, but we recall that a strong solution $u \in W^{2,p}_{\loc}$ with $p \geq 2n$ is a $C$-viscosity solution, see \cite[Section III]{Lio83}.

In order to apply this result we need two preliminary observations: firstly, we have to associate to our equation $F(\DDC u)=f$ an equation for the real Hessian $\bar{F}(\DDR u)=f$ and, secondly, this $\bar{F}$ has to be defined over the whole space $\Sym^{2n}$.
This is done in a standard way (see e.g. \cite{Blo99,Wan12-MRL,TWWY15}, etc.):

\begin{itemize}
\item
We identify $n\times n$ Hermitian matrices with the subspace of $\Sym^{2n}$ given by matrices invariant by the canonical complex structure:
$$
\iota(\Herm^n)
=\ens{A \in\Sym^{2n} \st AJ-JA=0},
\quad
J=
\begin{pmatrix}
0 & -\Id \\
\Id & 0
\end{pmatrix},
$$
where the map $\iota:\Herm^n \longrightarrow \Sym^{2n}$ is given by
$$
\iota(A+iB)=
\begin{pmatrix}
A & -B \\
B & A
\end{pmatrix}
.$$

Let us also introduce the projection $\pi: \Sym^{2n} \longrightarrow \iota(\Herm^n)$ given by
$$\pi(S)=\frac{S+J^{\tra} SJ}{2}.$$

The complex and real Hessians are then related by the identity
$$\iota(2D_{\C}^2 u)=\pi(D_{\R}^2u).$$

As a result, if $u$ solves $F(\DDC u)=f$, then it solves as well
$$\tilde{F}(\DDR u)=f,$$
where $\tilde{F}:\matconeR_R \longrightarrow \R$ is given by
$$\tilde{F}(A)=F\left(\frac{1}{2}\iota^{-1}\left(\pi(A)\right)\right),$$
and $\matconeR_R \subset \Sym^{2n}$ is defined by
$$\matconeR_R=\pi^{-1}\left(\iota\left(2\matcone_R\right)\right).$$ 
We recall that $\matcone_R$ is defined in \eqref{def CR} and that it is not a cone, so that the factor $2$ cannot be removed.

It is clear that $\matconeR_R$ is an open convex set and that $\tilde{F}$ is concave in $\matconeR_R$.
Since $F$ is uniformly elliptic in $\matcone_R$ by assumption of Corollary \ref{cor main thm}, it is also clear that $\tilde{F}$ is uniformly elliptic in $\matconeR_R$.

\item
Let us now extend $\tilde{F}$ to $\Sym^{2n}$.
This is done by considering $\bar{F}:\Sym^{2n} \longrightarrow \R$ defined by
$$\bar{F}(A)=\inf\ens{L(A) \st L:\Sym^{2n} \longrightarrow \R \text{ affine linear, } \quad m\Id \leq DL \leq M \Id, \quad L \geq \tilde{F} \text{ in } \matconeR_R},$$
where $0<m \leq M$ denote the ellipticity constants of $\tilde{F}$.
We can check that $\bar{F}$ is concave and uniformly elliptic in $\Sym^{2n}$, and that we have $\bar{F}=\tilde{F}$ in $\matconeR_R$ (see e.g. \cite[Lemma 4.1]{TWWY15}).
\end{itemize}

Corollary \ref{cor main thm} is now an easy consequence of the previous results.

\begin{proof}[Proof of Corollary \ref{cor main thm}]

After translation and dilation of the coordinates if necessary, it is sufficient to show that $u \in C^{\infty}(B_\epsilon)$ for some $\epsilon \in (0,1)$ when $B_1 \subset\subset B_2 \subset\subset \Omega$.

From Theorem \ref{main thm} we know that, for some $R_1>0$,
$$\Delta u <R_1 \quad \text{ in } B_2.$$ 
On the other hand, by our assumptions on $f$, there exists $R_2>0$ such that
$$F(\DDC u)=f>\frac{1}{R_2}.$$
Consequently, for any $R>\max\ens{R_0,R_1,R_2}$ we have
$$\DDC u \in \matcone_R.$$
It now follows from the previous discussion that $u$ is a strong solution to the concave and uniformly elliptic real equation
$$\bar{F}(\DDR u)=f \quad \mbox{ in } B_1.$$
Besides, $u \in W^{2,p}(B_1)$ for every $1 \leq p<\infty$ by Calder\'on-Zygmund estimates since $u, \Delta u \in L^{\infty}(B_2)$.
We can then apply Theorem \ref{thm CC95} and obtain that $u \in C^{2,\alpha}(B_\epsilon)$ for some $\alpha \in (0,1)$ and $\epsilon \in (0,1)$.

As a result, $u \in C^{2,\alpha}(B_\epsilon)$ solves $\tilde{F}(\DDR u)=f$ in $B_\epsilon$.
Since
$$F \in C^{\infty}(\matcone) \quad \Longrightarrow \quad \tilde{F} \in C^{\infty}(\matconeR_R),$$
it follows that $u \in C^{\infty}(B_\epsilon)$ from the classical Schauder's estimates, see e.g. \cite[Proposition 9.1]{CC95} (the proof in this reference is carried out for $\matconeR_R=\Sym^{2n}$ but all the arguments go through after noticing that $\ens{tD^2 u(x+he_k)+(1-t) D^2 u(x) \st t \in [0,1], \, x \in \clos{H}, \, 0\leq h \leq \dist(H,\partial\Omega)/2}$ is a compact set of $\Sym^{2n}$ which is included in $\matconeR_R$).

\end{proof}

\section{Application to Hessian equations}\label{sect examples}

Let us now present some examples covered by our framework.
We emphasize that our results Theorem \ref{main thm} and Corollary \ref{cor main thm} do not require that our equations are Hessian, but all the examples of application that we present here will be Hessian equations.

Let us first recall some notations.

\begin{itemize}
\item
For $A \in \Herm^n$, its eigenvalues (which are real) will always be sorted as follows:
$$\lambda_1(A) \leq \lambda_2(A) \leq \cdots \leq \lambda_n(A).$$
We then introduce $\lambda:\Herm^n \longrightarrow \R^n$ defined by
$$\lambda(A)=(\lambda_1(A),\ldots,\lambda_n(A)).$$

\item
A function $F:\matcone \longrightarrow \R$ is said to be a Hessian operator if there exist a set $\veccone \subset \R^n$ and a function $\hat{F}:\veccone \longrightarrow \R$ such that
$$\matcone=\lambda^{-1}(\veccone), \quad F(A)=\hat{F}(\lambda(A)), \quad \forall A \in \matcone.$$
The notation $\veccone$ will always be saved for sets of $\R^n$ whereas the notation $\matcone$ will always be saved for sets of $\Herm^n$.

\item
For $k \in \ens{1,\ldots,n}$, we introduce the classical cones
$$
\veccone_k=\ens{\lambda \in \R^n \st \sigma_\ell(\lambda_1,\ldots,\lambda_n)>0, \quad \forall \ell \in \ens{1,\ldots,k}},
$$
where $\sigma_k$ is the $k$-th elementary symmetric polynomial:
$$
\sigma_k(\lambda)=
\sum_{(i_1,\ldots,i_k) \in E_n^k} \lambda_{i_1} \cdots \lambda_{i_k},
$$
where, here and in the rest of this article, we use the notation
$$E_n^k=\ens{(i_1,\ldots,i_k) \in \ens{1,\ldots,n}^k \st i_1<\ldots<i_k}.$$
We recall that $\card E_n^k=C_n^k=n!/(k!(n-k)!)$.

\item
We recall that a set $S \subset \R^n$ is symmetric if $(\lambda_{\tau(1)},\ldots,\lambda_{\tau(n)}) \in S$ for every $(\lambda_1,\ldots,\lambda_n) \in S$ and every permutation $\tau$.
A function $\hat{F}:S \longrightarrow \R$ is symmetric if $\hat{F}(\lambda_{\tau(1)},\ldots,\lambda_{\tau(n)})=\hat{F}(\lambda_1,\ldots,\lambda_n)$ for every permutation $\tau$.

\end{itemize}

Let us now provide practical conditions on $\veccone$ and $\hat{F}$ to guarantee that $\matcone=\lambda^{-1}(\veccone)$ and $F=\hat{F} \circ \lambda$ satisfy the assumptions of our main results.

\begin{proposition}\label{prop hatF}
Let $\hat{F} \in C^{\infty}(\R^n,\R)$ be a symmetric polynomial, homogeneous of degree $\degr \in \N^*$, satisfying, for some $C>0$,
\begin{equation}\label{hyp comp with MA hatF}
\hat{F}(\lambda)^{1/\degr} \geq C \left(\prod_{i=1}^n \lambda_i\right)^{1/n}, \quad \forall \lambda \in \veccone_n,
\end{equation}
and let $\veccone \subset \R^n$ be a nonempty open symmetric convex cone such that
\begin{gather*}
\veccone_n \subset \veccone \subset \veccone_1, \quad \partial \veccone \cap \partial \veccone_1=\ens{0}, \\
\hat{F}>0 \text{ in } \veccone, \quad \hat{F}=0 \text{ on } \partial\veccone, \\
\hat{F}^{1/\degr} \text{ is concave in } \veccone.
\end{gather*}

Then, $\matcone=\lambda^{-1}(\veccone)$ satisfies \ref{domain} and $F=\hat{F} \circ \lambda$ satisfies \ref{reg f}, \ref{pos f}, \ref{F hom}, \ref{hyp cvx}, \ref{hyp comp with MA} and \ref{prop for UE}.
\end{proposition}

We emphasize that the map $\lambda$ has few nice properties, so this result is far from being trivial.
The proof of Proposition \ref{prop hatF} is postponed to the end of this section for the sake of the presentation.

\begin{remark}
The concavity assumption is satisfied by a large class of polynomials, namely, the hyperbolic polynomials.
We recall that a polynomial $\hat{F}$ homogeneous of degree $\degr \in \N^*$ is called hyperbolic with respect to $e \in \R^n$ if $\hat{F}(e) \neq 0$ and if, for every $\lambda \in \R^n$, the one-variable polynomial $t \in \C \longmapsto \hat{F}(\lambda-t e)$ has only real roots.
In this case, $\hat{F}^{1/\degr}$ is concave in the open convex cone (called hyperbolicity cone)
$$\Gamma=\ens{\lambda \in \R^n \st \hat{F}(\lambda-te)\neq 0, \quad \forall t \leq 0}.$$
We refer to \cite[Section 1]{CNS85} and \cite{Gar59} for more details.
\end{remark}

Let us now finally present some concrete examples.

\begin{itemize}
\item
\textbf{The complex Monge-Ampère equation:}
$$
\hat{F}(\lambda)=\prod_{i=1}^n \lambda_i,
\quad
\veccone=\veccone_n.
$$
The degree of homogeneity is obviously $\degr=n$ and the comparison property \eqref{hyp comp with MA hatF} is trivial.
This polynomial is hyperbolic with respect to $e=(1,1,\ldots,1)$.
We can check that the associated hyperbolicity cone is indeed $\veccone_n$.
It clearly satisfies all the desired assumptions.

\item
\textbf{The complex $k$-Hessian equation:} for $k \in \ens{1,\ldots,n}$,
$$
\hat{F}=\sigma_k,
\quad
\veccone=\veccone_k.
$$
The degree of homogeneity is obviously $\degr=k$ and the comparison property \eqref{hyp comp with MA hatF} follows from Maclaurin's inequality: for any $k \geq 2$,
$$
\veccone_{k-1} \subset \veccone_k,
\quad
\left(\frac{\sigma_k}{C_n^k}\right)^{1/k} \leq \left(\frac{\sigma_{k-1}}{C_n^{k-1}}\right)^{1/(k-1)} \mbox{ in } \veccone_{k-1}.
$$
This polynomial is hyperbolic with respect to $e=(1,1,\ldots,1)$.
We can check that the associated hyperbolicity cone is indeed $\veccone_k$.
It clearly satisfies all the desired assumptions ($k>1$).

\item
For $s \in [0,1)$, let
\begin{gather*}
\begin{array}{rl}
\hat{F}(\lambda) &=(1-s)^2\lambda_1 \lambda_2+s(\lambda_1+\lambda_2)^2
\\
&=(\lambda_1+s\lambda_2)(s\lambda_1+\lambda_2),
\end{array}
\\
\veccone=\Gamma_{2-s}=\ens{\lambda \in \R^2 \st \lambda_1+s\lambda_2>0, \quad s\lambda_1+\lambda_2>0}.
\end{gather*}
The cones $\Gamma_{2-s}$ interpolate between $\Gamma_2$ and $\Gamma_1$.
The degree of homogeneity is obviously $\degr=2$ and the comparison property \eqref{hyp comp with MA hatF} is trivial since $\hat{F}(\lambda) \geq (1-s)^2\lambda_1 \lambda_2$.
The cone clearly satisfies all the desired assumptions ($\hat{F}$ is also a hyperbolic polynomial with respect to $e=(1,1)$ but the concavity of $\hat{F}^{1/2}$ is easily checked by a direct computation here).

\item\label{ex Mk}
\textbf{The complex $k$-Monge-Ampère equation:} for $k \in \ens{1,\ldots,n}$,
$$
\hat{F}=\MA_k,
\quad
\veccone=\veccone_k'.
$$
where
\begin{gather*}
\MA_k(\lambda)=
\prod_{(i_1,\ldots,i_k) \in E_n^k} \left(\lambda_{i_1}+\ldots+\lambda_{i_k}\right),
\\
\veccone_k'=
\ens{\lambda \in \R^n \st \lambda_{i_1}+\ldots+\lambda_{i_k}>0, \quad \forall (i_1,\ldots,i_k) \in E_n^k}.
\end{gather*}

The degree of homogeneity is $\degr=C_n^k$.
The comparison property \eqref{hyp comp with MA hatF} is a consequence of the following inequality: for any $k \geq 2$,
\begin{equation}\label{ineq MAk}
\veccone_{k-1}' \subset \veccone_k',
\quad
\frac{1}{k}\left(\MA_k\right)^{1/C_n^k} \geq \frac{1}{k-1}\left(\MA_{k-1}\right)^{1/C_n^{k-1}} \mbox{ in } \veccone_{k-1}'.
\end{equation}
This property will be detailed in Section \ref{sect ineq MAk} below.
This polynomial is again hyperbolic with respect to $e=(1,1,\ldots,1)$.
We can check that the associated hyperbolicity cone is indeed $\veccone_k'$.
It clearly satisfies all the desired assumptions ($k<n$).

\end{itemize}

We conclude this section with the proof Proposition \ref{prop hatF}.

\begin{proof}[Proof of Proposition \ref{prop hatF}]
We only establish the non obvious properties.
\begin{enumerate}[1)]
\item
\textbf{Regularity of $F$.}
Since $\hat{F}$ is a symmetric polynomial, by Newton's fundamental theorem of symmetric polynomials we can write
$$\hat{F}(\lambda)=p(\sigma_1(\lambda),\ldots,\sigma_n(\lambda)),$$
for some polynomial $p$.
Since all the functions $A \in \Herm^n \longmapsto \sigma_k(\lambda(A))$ are $C^{\infty}(\Herm^n,\R)$ (recall that they are equal to $(-1)^k (\partial^{n-k} h/\partial t^{n-k})(0,A)$ where $h \in C^{\infty}(\R \times \Herm^n,\R)$ is the function $h(t,A)=\det(t\Id-A)$), we deduce that $F \in C^{\infty}(\Herm^n,\R)$ by composition.

\item
\textbf{Convexity of $\matcone$ and concavity of $F^{1/\degr}$.}
We point out that the concavity of $F^{1/\degr}$ is also shown in \cite[Section 3]{CNS85} but we would like to present a different proof here.
Let $A,B \in \matcone$ and $t \in [0,1]$ be fixed.
We have to show that
\begin{gather*}
\lambda(tA+(1-t)B) \in \veccone, \\
F(\lambda(tA+(1-t)B))^{1/\degr} \geq tF(\lambda(A))^{1/\degr}+(1-t)F(\lambda(B))^{1/\degr}.
\end{gather*}
To this end, it is convenient to make use of the theory of majorization, for which we refer to \cite{MOA11}.
The starting point is that, for every $k \in \ens{1,\ldots,n}$, the function
$$
A \in \Herm^n \longmapsto \lambda_1(A)+\ldots+\lambda_k(A) \text{ is concave.}
$$
From this property, we have that $t\lambda(A)+(1-t)\lambda(B)$ ``majorizes'' $\lambda(tA+(1-t)B)$, that is
\begin{equation}\label{majorization}
\begin{dcases}
\sum_{j=1}^k \lambda_j(tA+(1-t)B)
\geq \sum_{j=1}^k \left(t\lambda_j(A)+(1-t)\lambda_j(B)\right), \quad \forall k \in \ens{1,\ldots,n-1},
\\
\sum_{j=1}^n \lambda_j(tA+(1-t)B)
=\Trace(tA+(1-t)B)
=\sum_{j=1}^n \left(t\lambda_j(A)+(1-t)\lambda_j(B)\right).
\end{dcases}
\end{equation}
It follows that (see e.g. \cite[Chapter 1, B.1. Lemma]{MOA11})
$$\lambda(tA+(1-t)B)^\tra
=T_N T_{N-1}\cdots T_1 (t\lambda(A)+(1-t)\lambda(B))^\tra,$$
for some matrices $T_1, \ldots, T_N \in \R^{n \times n}$ of the form
$$T_i=t_i \Id+ (1-t_i)P_i,$$
for some $t_i \in [0,1]$ and some elementary permutation matrix $P_i \in \R^{n \times n}$ (i.e. switching only two components).
Since $\veccone$ is convex and symmetric, it is clear that such transformations map $\veccone$ into itself.
As a result, we have $\lambda(tA+(1-t)B) \in \veccone$.

On the other hand, since $\hat{F}^{1/\degr}$ is concave and symmetric, it is Schur-concave (see e.g. \cite[Chapter 3, C.2. Proposition]{MOA11}), meaning that the majorization \eqref{majorization} implies that
$$\hat{F}(t\lambda(A)+(1-t)\lambda(B))^{1/\degr} \leq \hat{F}(\lambda(tA+(1-t)B))^{1/\degr}.$$
Using again the concavity of $\hat{F}^{1/\degr}$, we obtain the concavity of $F^{1/\degr}$.

\item
\textbf{Uniform ellipticity of $F$ in $\matcone_R$.}
Let $R>0$ be fixed.
Let
$$
\veccone_R=\ens{\lambda \in \veccone \st  \sum_{i=1}^n \lambda_i<R, \quad
\frac{1}{R}<\hat{F}(\lambda)}.
$$
We first show that the assumption $\partial \veccone \cap \partial \veccone_1=\ens{0}$ implies that $\veccone_R$ is bounded.
To this end, it is sufficient to prove that there exists $\epsilon>0$ such that
$$\veccone \subset \ens{\lambda \in \veccone \st  \epsilon \abs{\lambda}<\sum_{i=1}^n \lambda_i}.$$
To show this property, we argue by contradiction and assume that there exists a sequence $(\lambda^\epsilon)_{\epsilon>0} \subset \veccone$ such that
\begin{equation}\label{contra}
\epsilon \abs{\lambda^\epsilon} \geq \sum_{i=1}^n \lambda^\epsilon_i.
\end{equation}
Since $\veccone \subset \veccone_1$, we have in particular that $\lambda^\epsilon \neq 0$ and we can normalize the sequence by considering $\tilde{\lambda}^\epsilon=\lambda^\epsilon/\abs{\lambda^\epsilon}$.
We have $\tilde{\lambda}^\epsilon \in \veccone$ since $\veccone$ is a cone.
We can extract a subsequence, still denoted by $(\tilde{\lambda}^\epsilon)_{\epsilon>0}$, that converges to some $\tilde{\lambda} \in \clos{\veccone}$.
Besides, $\tilde{\lambda} \neq 0$ since $\abs{\tilde{\lambda}}=1$.
Passing to the limit $\epsilon \to 0$ in \eqref{contra}, we obtain that
$$0 \geq \sum_{i=1}^n \tilde{\lambda}_i.$$
Since $\veccone \subset \veccone_1$, it follows that $\tilde{\lambda} \in \partial \veccone \cap \partial \veccone_1$ and thus $\tilde{\lambda}=0$ by assumption, a contradiction.

As a result, $\clos{\veccone_R}$ is compact and there exist real numbers $m \leq M$ such that, for every $i \in \ens{1,\ldots,n}$,
\begin{equation}\label{UE hatF}
m \leq \frac{\partial \hat{F}}{\partial \lambda_i}(\lambda) \leq M, \quad \forall \lambda \in \clos{\veccone_R}.
\end{equation}

Let us now prove that $m>0$.
To this end, it is sufficient to show that
\begin{gather*}
0<\frac{\partial \hat{F}}{\partial \lambda_i}(\lambda), \quad \forall \lambda \in \veccone, \quad \forall i \in \ens{1,\ldots,n},
\\
\clos{\veccone_R} \subset \veccone.
\end{gather*}

The inclusion $\clos{\veccone_R} \subset \veccone$ immediately follows from the assumption that $\hat{F}=0$ on $\partial\veccone$.

Let us now prove the inequality.
We use some ideas of the proof of \cite[Corollary, pp. 269-270]{CNS85}.
Let $i \in \ens{1,\ldots,n}$ be fixed.
Let $e_i$ denote the $i$-th canonical vector of $\R^n$.
We have $e_i \in \clos{\veccone_n} \subset \clos{\veccone}$.
From the concavity and homogeneity of $\hat{G}=\hat{F}^{1/\degr}$, which belongs to $C^1(\veccone) \cap C^0(\clos{\veccone})$ since $\hat{F}>0$ in $\veccone$, we have
$$0 \leq \hat{G}(e_i)
\leq \hat{G}(\lambda)+D\hat{G}(\lambda)(e_i-\lambda)
=D\hat{G}(\lambda)e_i
=\frac{\partial \hat{G}}{\partial \lambda_i}(\lambda).
$$
Let us now prove that this inequality is strict.
Assume not, and let then $\lambda^* \in \veccone$ be such that $(\partial \hat{G}/ \partial \lambda_i)(\lambda^*)=0$.
Since $\veccone$ is open, there exists $\epsilon>0$ such that $\lambda^*+t e_i \in \veccone$ for every $t \in [0,\epsilon]$.
By concavity, we then have
$$
\left(D\hat{G}(\lambda^*+te_i)-D\hat{G}(\lambda^*)\right)(te_i) \leq 0, \quad \forall t \in [0,\epsilon].
$$
As a result,
$$\frac{\partial \hat{G}}{\partial \lambda_i}(\lambda^*+te_i)=0, \quad \forall t \in [0,\epsilon].$$
Since $\hat{F}>0$ in $\veccone$, this is equivalent to
$$\frac{\partial \hat{F}}{\partial \lambda_i}(\lambda^*+te_i)=0, \quad \forall  t \in [0,\epsilon].$$
By analyticity, we deduce that this identity holds for any $t \in \R$, in particular for negative ones.
Let then
$$T(\lambda^*)=\inf\ens{t \in \R \st \lambda^*+te_i \in \veccone}.$$
Since $\veccone \subset \veccone_1$, we necessarily have $T(\lambda^*)>-\infty$.
Integrating the identity $(\partial \hat{F}/\partial \lambda_i)(\lambda^*+te_i)=0$ over $[T(\lambda^*),0]$ then yields
$$\hat{F}(\lambda^*)=\hat{F}(\lambda^*+T(\lambda^*)e_i).$$
This is a contradiction since $\hat{F}>0$ in $\veccone \ni \lambda^*$, whereas $\hat{F}=0$ on $\partial\veccone \ni \lambda^*+T(\lambda^*)e_i$.

Therefore, \eqref{UE hatF} holds with $m>0$.
To conclude, it remains to show that this implies the uniform ellipticity of $F=\hat{F} \circ \lambda$ in $\matcone_R$.
Let then $A,B \in \matcone_R$ with $A \leq B$.
We denote the eigenvalues of $A$ (resp. $B$) by $(\lambda_1,\ldots,\lambda_n)$ (resp. $(\mu_1,\ldots,\mu_n)$).
By definition, we have to show that
$$m'\norm{B-A}_{\Herm^n} \leq F(B)-F(A) \leq M'\norm{B-A}_{\Herm^n},$$
for some $0<m' \leq M'$ (that do not depend on $A,B$).
We only prove the first inequality, the other one being proved similarly.
Since $B-A \geq 0$, the first inequality is equivalent to
$$m''\Trace(B-A) \leq F(B)-F(A),$$
for some $0<m''$.
Since $\Trace(B-A)=\Trace(B)-\Trace(A)$, this is equivalent to the following property for $\hat{F}$:
\begin{equation}\label{caract UE hatF}
m''\sum_{i=1}^n (\mu_i-\lambda_i) \leq \hat{F}(\mu)-\hat{F}(\lambda).
\end{equation}

Let us then prove this inequality.
By definition, we have $\lambda, \mu \in \veccone_R$.
Assume first that
\begin{equation}\label{easy case}
(\mu_1,\ldots,\mu_{n-1},\lambda_n) \in \veccone_R, \quad (\mu_1,\ldots,\mu_{n-2},\lambda_{n-1},\lambda_n) \in \veccone_R, \quad \ldots \quad (\mu_1,\lambda_2,\ldots,\lambda_n) \in \veccone_R.
\end{equation}
Then, taking $t(\mu_1,\ldots,\mu_{n-1},\mu_n)+(1-t)(\mu_1,\ldots,\mu_{n-1},\lambda_n) \in \veccone_R$ in \eqref{UE hatF} and integrating over $t \in [0,1]$, we have
$$
m(\mu_n-\lambda_n) \leq \hat{F}(\mu_1,\ldots,\mu_{n-1},\mu_n)-\hat{F}(\mu_1,\ldots,\mu_{n-1},\lambda_n).
$$
Taking then $t(\mu_1,\ldots,\mu_{n-1},\lambda_n)+(1-t)(\mu_1,\ldots,\mu_{n-2},\lambda_{n-1},\lambda_n) \in \veccone_R$ in \eqref{UE hatF}, iterating this process and summing all the obtained inequalities, we eventually obtain \eqref{caract UE hatF} with $m''=m$.
The general case can be deduced from the case \eqref{easy case} by an approximation argument (using the compactness of $\clos{\veccone_R}$).

\end{enumerate}

\end{proof}

\section{Additional properties for the $\MA_k$-equation}\label{sect Mk}

In this section, we detail the case of the $k$-Monge-Ampère equation, which is important for applications.

Let us first state explicitly the local regularity result that we have obtained for this equation.

\begin{theorem}\label{main thm Mk}
Let $k \in \ens{1,\ldots,n-1}$ ($n \geq 2$).
Let $f \in C^{\infty}(\Omega)$ with $f>0$ in $\Omega$.
Let $u \in W^{2,p}_{\loc}(\Omega)$ with $(\DDC u)(z) \in \matcone_k'=\lambda^{-1}(\veccone_k')$ for a.e. $z \in \Omega$ satisfy (almost everywhere)
$$\MA_k(\lambda(\DDC u))=f \quad \mbox{ in } \Omega.$$
If $p>n(C_n^k-1)$, then $u \in C^{\infty}(\Omega)$.
\end{theorem}

We recall once again that, unless $k=1$ (the Monge-Ampère equation), it is not known if this condition on $p$ is sharp.
However, we will provide in Section \ref{sect lower bound} below an example that shows that a threshold does exist for such a result to be valid.

\subsection{$k$-plurisubharmonic functions}

The importance in the study of the operator $\MA_k$ lies in its connection with the notion of $k$-plurisubharmonic function.

\begin{definition}
A function $u:\Omega \longrightarrow \R$ is said to be $k$-plurisubharmonic in $\Omega$ if it is upper semi-continuous and if it is subharmonic whenever it is restricted to any affine complex plane of dimension $k$.
We denote by $k$-plurisubharmonic the vector space of such functions.
\end{definition}

For instance, $1$-plurisubharmonic functions are the plurisubharmonic functions, and $n$-pluri-subharmonic functions are the subharmonic functions.

To make clearer the link between $k$-plurisubharmonic functions and the operator $\MA_k$, we first introduce some notations.

\begin{definition}
Let $A \in \C^{n \times n}$ and $k \in \ens{1,\ldots,n}$.
The operator $D_A: \Lambda^k \C^n \longrightarrow \Lambda^k \C^n $ is the linear action of $A$ as derivation on the space $\Lambda^k \C^n$ of $k$-vectors.
On simple $k$-vectors, this means that 
\begin{multline*}
D_A(v_1\wedge v_2 \wedge \cdots \wedge v_k)=
(Av_1)\wedge v_2 \wedge \cdots\wedge v_k
+v_1\wedge (Av_2) \wedge \cdots \wedge v_k
+\cdots
\\
+v_1\wedge v_2 \wedge \cdots \wedge (Av_k),
\end{multline*}
for every $v_1,\ldots, v_k \in \C^n$.
\end{definition}

By choosing the canonical basis of the complex space $\Lambda^k \C^n$ (recall that $\dim \Lambda^k \C^n=C_n^k$), $D_A$ can be identified with a matrix of size $C_n^k \times C_n^k$, that will still be denoted by $D_A$.

\begin{example}
Let $A=\left(u_{i \bar{j}}\right)_{1 \leq i,j \leq n}$.
\begin{itemize}
\item
For $n=3$ and $k=2$, we have
$$
D_A=
\begin{pmatrix}
u_{1\bar{1}}+u_{2\bar{2}}& u_{2\bar{3}} & -u_{1\bar{3}}\\
u_{3\bar{2}} & u_{1\bar{1}}+u_{3\bar{3}} & u_{1\bar{2}}\\
-u_{3\bar{1}} & u_{2\bar{1}} & u_{2\bar{2}}+u_{3\bar{3}}
\end{pmatrix}.
$$

\item
For $n=4$ and $k=3$, we have
$$
D_A=
\begin{pmatrix}
u_{1\bar{1}}+u_{2\bar{2}}+u_{3\bar{3}}& u_{3\bar{4}} & u_{1\bar{4}}& -u_{2\bar{4}}\\

u_{4\bar{3}} & u_{1\bar{1}}+u_{2\bar{2}}+u_{4\bar{4}} &- u_{1\bar{3}}&u_{2\bar{3}} \\

u_{4\bar{1}} & -u_{3\bar{1}} & u_{2\bar{2}}+u_{3\bar{3}}+u_{4\bar{4}}& u_{2\bar{1}}\\

-u_{4\bar{2}} & u_{3\bar{2}} & u_{1\bar{2}}&  u_{3\bar{3}}+u_{4\bar{4}}+u_{1\bar{1}}
\end{pmatrix}.
$$

\item
For $n=4$ and $k=2$, we have
$$
D_A=
\begin{pmatrix}
u_{1\bar{1}}+u_{2\bar{2}}& u_{2\bar{3}} & u_{2\bar{4}}& -u_{1\bar{3}}& -u_{1\bar{4}}& 0\\

u_{3\bar{2}} & u_{1\bar{1}}+u_{3\bar{3}} & u_{3\bar{4}}& u_{1\bar{2}}&0 & -u_{1\bar{4}}\\

u_{4\bar{2}} & u_{4\bar{3}} & u_{1\bar{1}}+u_{4\bar{4}}& 0 & u_{1\bar{2}} & u_{1\bar{3}} \\

-u_{3\bar{1}} & u_{2\bar{1}} & 0& u_{2\bar{2}}+u_{3\bar{3}}&u_{3\bar{4}} & -u_{2\bar{4}}\\

-u_{4\bar{1}} & 0 & u_{2\bar{1}}& u_{4\bar{3}}&u_{2\bar{2}}+u_{4\bar{4}} & u_{2\bar{3}} \\

0 & -u_{4\bar{1}} & u_{3\bar{1}}& -u_{4\bar{2}}& u_{3\bar{2}}& u_{3\bar{3}}+u_{4\bar{4}}
\end{pmatrix}.
$$

\end{itemize}
\end{example}

We can check that $D_A$ is Hermitian if so is $A$ and that the spectrum of $D_A$ is exactly
$$\ens{\lambda_{i_1}(A)+\ldots+\lambda_{i_k}(A) \st (i_1,\ldots,i_k) \in E^k_n},$$
so that it becomes clear that
$$\MA_k(\lambda(A))=\det (D_A),$$
and
$$\matcone_k'=\ens{A \in \Herm^n \st D_A>0}.$$
This cone is linked to the notion of $k$-plurisubharmonic function as follows: when $u \in C^2(\Omega)$, we have the characterization
$$u \text{ is $k$-plurisubharmonic in } \Omega \quad \Longleftrightarrow \quad (\DDC u)(z) \in \clos{\matcone_k'}, \quad \forall z \in \Omega.$$

For more information about $k$-plurisubharmonic functions we refer for instance the reader to \cite{HL11,HL13}.

\subsection{Comparison between $\MA_k$ and $\MA_{k-1}$}\label{sect ineq MAk}

In this section, we show that the crucial comparison property \eqref{hyp comp with MA hatF} holds for $\MA_k$.
In fact, we prove the more precise property \eqref{ineq MAk} that allows to compare $\MA_k$ and $\MA_{k-1}$.

Let $k \in \ens{2,\ldots,n}$ be fixed.
Let us first show the inclusion
$$\veccone_{k-1}' \subset \veccone_k'.$$
Let then $\lambda \in \veccone_{k-1}'$ and let us show that $\lambda_{i_1}+\ldots+\lambda_{i_k}>0$.
For any $(i_1,\ldots,i_k) \in E^k_n$, we denote by $I_k=\ens{i_1,\ldots,i_k}$ and
$$E^{k-1}_{I_k}=\ens{(j_1,\ldots,j_{k-1}) \st j_1,\ldots,j_{k-1} \in I_k, \quad j_1<\ldots<j_{k-1}}.$$
Then, the claim follows from the identity
\begin{equation}\label{estim cones}
0<\sum_{(j_1,\ldots,j_{k-1}) \in E^{k-1}_{I_k}} \left(\lambda_{j_1}+\ldots+\lambda_{j_{k-1}}\right)
=(k-1)\left(\lambda_{i_1}+\ldots+\lambda_{i_k}\right).
\end{equation}
To show this identity, we have to count how many $\lambda_{i_1}, \lambda_{i_2},\ldots$ are in the sum of the left-hand side.
For $\lambda_{i_1}$ the only possibility is to have $j_1=i_1$.
Therefore, there are $C_{k-2}^{k-1}=k-1$ such terms.
For $\lambda_{i_2}$ there are two and only two possibilities, namely $j_1=i_2$ or $j_2=i_2$ (in this second case, necessarily $j_1=i_1$).
This gives $C_{k-3}^{k-1}+C_{k-3}^{k-2}$ terms, which is equal to $C_{k-2}^{k-1}$ by Pascal's rule.
Repeating this reasoning leads to \eqref{estim cones}.

Let us now show the inequality
$$\MA_k(\lambda)\leq \MA_{k-1}(\lambda),$$
for $\lambda \in \veccone_{k-1}'$.
This is equivalent to show that
$$
\prod_{(i_1,\ldots,i_k) \in E_n^k} \left(\lambda_{i_1}+\ldots+\lambda_{i_k}\right)
\geq
\left(\frac{k}{k-1}\right)^{C_n^k}
\left(
\prod_{(j_1,\ldots,j_{k-1}) \in E^{k-1}_{I_k}}
\left(\lambda_{j_1}+\ldots+\lambda_{j_{k-1}}\right)
\right)^{C_n^k /C_n^{k-1}}.
$$
Using \eqref{estim cones} and the inequality of arithmetic and geometric means (note that $\card E^{k-1}_{I_k}=k$), we have
$$
\begin{array}{rl}
\ds
\lambda_{i_1}+\ldots+\lambda_{i_k}
&\ds =\frac{k}{k-1}
\frac{\sum_{(j_1,\ldots,j_{k-1}) \in E^{k-1}_{I_k}}  \left(\lambda_{j_1}+\ldots+\lambda_{j_{k-1}}\right)}{k}
\\
&\ds \geq 
\frac{k}{k-1}
\left(
\prod_{(j_1,\ldots,j_{k-1}) \in E^{k-1}_{I_k}}
\left(\lambda_{j_1}+\ldots+\lambda_{j_{k-1}}\right)
\right)^{1/k}.
\end{array}
$$
To conclude, it remains to observe the identity
\begin{multline*}
\prod_{(i_1,\ldots,i_k) \in E^k_n}
\left(\prod_{(j_1,\ldots,j_{k-1}) \in E^{k-1}_{I_k}}
\left(\lambda_{j_1}+\ldots+\lambda_{j_{k-1}}\right)
\right)
\\
=
\left(\prod_{(j_1,\ldots,j_{k-1}) \in E^{k-1}_n}
\left(\lambda_{j_1}+\ldots+\lambda_{j_{k-1}}\right)
\right)^{kC_n^k /C_n^{k-1}},
\end{multline*}
which can be established as before by counting the terms $\lambda_{j_1}+\ldots+\lambda_{j_{k-1}}$ that appear in the left-hand side (note that $kC_n^k/C_n^{k-1}=n-k+1$).
\qed

\subsection{A lower bound for $p$}\label{sect lower bound}

Let us conclude Section \ref{sect Mk} by mentioning that, even if the condition $p>n(C_n^k-1)$ is not known to be sharp, it is however necessary for the conclusion of Theorem \ref{main thm Mk} to be true that $p$ is not too small.
More precisely, we will show that it cannot be smaller than
$$p^*=(n-k)C_n^k.$$
Note that $p^*$ is only equal to $n(C_n^k-1)$ for $k=1$ (and $k=n$ but we have excluded this value).
This follows from the following type of Pogorelov's example, which is built on the one of \cite{Blo99}.

\begin{proposition}
Let $1 \leq m<n$ and $\beta>0$.
The function
$$u(z',z'')=\left(1+\abs{z'}^2\right) \abs{z''}^{2\beta},$$
where $(z',z'') \in \C^m \times \C^{n-m}$, satisfies:
\begin{enumerate}[(i)]
\item\label{Pog ex calcul Hess}
$u \in C^{\infty}(\C^n \backslash N)$, where $N=\ens{z \in \C^n \st z''=0}$, and $(\DDC u)(z)>0$ for every $z \in \C^n \backslash N$.

\item\label{Pog ex w2p}
$u \in W^{2,p}_{\loc}(\C^n)$ ($p \geq 1$) if, and only if,
\begin{equation}\label{cond w2p}
\beta \geq 1 \quad \mbox{ or } \quad \left(\beta<1 \quad \mbox{ and } \quad p<\frac{n-m}{1-\beta}\right).
\end{equation}

\item
$u \in C_{\loc}^{1,\alpha}(\C^n)$ ($\alpha \in [0,1]$) if, and only if, $\alpha \leq 2\beta-1$.

\item\label{Pog ex calcul MAk}
For $\beta=1-C_m^k/C_n^k$ ($C_m^k=0$ if $m<k$), we have $\MA_k(\lambda(\DDC u)) \in C^{\infty}(\C^n)$ with $\MA_k(\lambda(\DDC u))>0$ in $\C^n$.

\end{enumerate}

\end{proposition}

For $\beta=1-1/C_n^k$, the largest $p$ that we can obtain from the condition \eqref{cond w2p} is when $m=k$ and it gives $p<(n-k)C_n^k$.
Consequently, we have constructed an example that satisfies
$$
\renewcommand\arraystretch{2}
\left.\begin{array}{c}
u \in W^{2,p}_{\loc}(\C^n), \quad \forall p \in \left[1,(n-k)C_n^k\right), \\
\DDC u>0 \mbox{ in } \C^n \backslash \ens{z \in \C^n \st z''=0}, \\
\MA_k(\lambda(\DDC u)) \in C^{\infty}(\C^n), \\
\MA_k(\lambda(\DDC u))>0 \mbox{ in } \C^n,
\end{array}\right\}
\quad \mbox{ but } \quad
u \not\in C_{\loc}^{1,\alpha}(\C^n), \quad \forall \alpha>1-\frac{2}{C_n^k}.
$$

\begin{remark}
The same function can be used to obtain lower bounds for various degenerate elliptic equations (Monge-Ampère, $k$-Hessian equation, etc.), although not always optimal.
\end{remark}

\begin{proof}[Proof of \ref{Pog ex calcul Hess}]
It is clear that $\DDC u$ is smooth off $N$.
Let $z \in \C^n \backslash N$.
A direct computation shows that we have the block decomposition
\begin{equation}\label{H expression}
\renewcommand\arraystretch{2}
(\DDC u)(z)
=
\left(\begin{array}{c|c}
\delta_{ij}\abs{z''}^{2\beta} & \beta \bar{z}_i z_j\abs{z''}^{2\beta-2} \\
\hline
\beta \bar{z}_i z_j \abs{z''}^{2\beta-2} & \beta\left(1+\abs{z'}^2\right)\left((\beta-1)\bar{z}_i z_j \abs{z''}^{2\beta-4}+\delta_{ij} \abs{z''}^{2\beta-2}\right)
\end{array}\right),
\end{equation}
where $\delta_{ij}$ denotes the Kronecker delta and $1 \leq i,j \leq m$ for the left upper block, $1 \leq i \leq m$ and $m+1 \leq j \leq n$ for the right upper block, etc.

Let us now find all the eigenvalues of $(\DDC u)(z)$.
From this expression we see that:
\begin{itemize}
\item
$\abs{z''}^{2\beta}$ is an eigenvalue of multiplicity at least $m-1$.
Indeed, since the right upper block defines a matrix of rank at most $1$, the rank of $\DDC u-\abs{z''}^{2\beta} \Id$ is thus at most $1+n-m$.
By the rank-nullity theorem, $\abs{z''}^{2\beta}$ is an eigenvalue of multiplicity at least $n-(1+n-m)=m-1$.

\item
Similarly, $\beta\left(1+\abs{z'}^2\right)\abs{z''}^{2\beta-2}$ is an eigenvalue of multiplicity at least $n-m-1$.
Besides, in case of equality with the first type of eigenvalue, it is then of multiplicity $n-2$.
\end{itemize}

Therefore, in every case we know $n-2$ eigenvalues and to find the two remaining ones, we are going to find their sum $S$ and product $P$.
The computation of $S$ is easy.
Indeed, the sum of the eigenvalues is
$$\Trace(\DDC u)=S+(m-1)\abs{z''}^{2\beta}+(n-m-1)\beta\left(1+\abs{z'}^2\right)\abs{z''}^{2\beta-2}.$$
On the other hand, from \eqref{H expression}, we have
$$
\Trace(\DDC u)
=m\abs{z''}^{2\beta}+\beta\left(1+\abs{z'}^2\right)(n-m+\beta-1)\abs{z''}^{2\beta-2}.
$$
Therefore,
$$S=\abs{z''}^{2\beta}+\beta^2 \left(1+\abs{z'}^2\right)\abs{z''}^{2\beta-2}.$$

In order to find  the product $P$, we will use the relation with the determinant of $\DDC u$.
As product of the eigenvalues, we have
\begin{equation}\label{det as p vp}
\begin{aligned}
\det(\DDC u) &=P\left(\abs{z''}^{2\beta}\right)^{m-1} \left(\beta\left(1+\abs{z'}^2\right)\abs{z''}^{2\beta-2}\right)^{n-m-1}
\\
&=P \beta^{n-m-1} \left(1+\abs{z'}^2\right)^{n-m-1}
\abs{z''}^{2\beta(n-2)-2(n-m-1)}.
\end{aligned}
\end{equation}
Let us now compute $\det(\DDC u)$ directly from \eqref{H expression}.
Doing the line substitutions
$$L_i \leftarrow L_i -\sum_{r=1}^m \beta \bar{z}_i z_r \abs{z''}^{-2} L_r,$$
for each $i \in \ens{m+1,\ldots,n}$, we obtain
$$
\renewcommand\arraystretch{2}
\det(\DDC u)
=
\det\left(\begin{array}{c|c}
\delta_{ij}\abs{z''}^{2\beta} & \beta \bar{z}_i z_j\abs{z''}^{2\beta-2} \\
\hline
0 & a_{ij}
\end{array}\right),
$$
where we introduced
$$a_{ij}=\beta\left(\beta-\left(1+\abs{z'}^2\right)\right)\bar{z}_i z_j \abs{z''}^{2\beta-4}
+\beta \delta_{ij}\left(1+\abs{z'}^2\right)\abs{z''}^{2\beta-2}.$$
Denoting by $A=(a_{ij})_{m+1 \leq i,j \leq n}$ we have
$$\det(\DDC u)=\left(\abs{z''}^{2\beta}\right)^m \det(A).$$
To compute $\det(A)$, we note that it is the sum of a rank one matrix with an invertible matrix, for which we have the formula
$$\det(J+v_1 v_2^{\tra})=\det(J)(1+v_2^{\tra} J^{-1}v_1),$$
for any invertible $J$ and vectors $v_1,v_2$.
Using this formula we obtain
$$\det(A)=
\left(\beta \left(1+\abs{z'}^2\right)\abs{z''}^{2\beta-2}\right)^{n-m}
\frac{\beta}{1+\abs{z'}^2}.$$
In summary,
$$
\det(\DDC u)=\beta^{n-m+1}
\left(1+\abs{z'}^2\right)^{n-m-1}
\abs{z''}^{2\beta n-2(n-m)}.
$$
Comparing with \eqref{det as p vp}, we finally obtain that
$$P=\beta^2 \abs{z''}^{4\beta-2}.$$

From the knowledge of the sum $S$ and product $P$, the two remaining eigenvalues are given by
$$
\frac{S+\sqrt{S^2-4P}}{2}
=\phi(z) \abs{z''}^{2\beta-2}
\quad \mbox{ and } \quad
\frac{2P}{S+\sqrt{S^2-4P}}
=\frac{\beta^2}{\phi(z)} \abs{z''}^{2\beta},
$$
where $\phi$ is the smooth positive function defined by
$$\phi(z)=\frac{\abs{z''}^2+\beta^2 \left(1+\abs{z'}^2\right)
+\sqrt{
\left(\abs{z''}^2+\beta^2 \left(1+\abs{z'}^2\right)\right)^2
-4\beta^2 \abs{z''}^2
}
}{2}.$$
Finally, note that all the eigenvalues are positive outside $N$.
\end{proof}

\begin{proof}[Proof of \ref{Pog ex w2p}]
We only investigate the integrability of the second order derivatives since we can check afterwards that the first order derivatives yield a less restrictive condition on the exponent of integrability $p$.
We recall that the only problem for the integrability of $u_{i\bar{j}}$ is near $N$, see the expression \eqref{H expression}.
From this expression, we also see that the ``worst term'' is then
$$\beta\left(1+\abs{z'}^2\right) (\beta-1)\bar{z}_i z_j\abs{z''}^{2\beta-4},$$
(unless $\beta=1$ but in such a case we clearly have $u_{i\bar{j}} \in C^0(\C^n)$ for every $i,j$).

Let us then consider the ball $B_{R}' \times B_{\epsilon}''$, with $R,\epsilon>0$, where $B_r'$ (resp. $B_r''$) denotes the open ball of $\C^m$ (resp. $\C^{n-m}$) with center $0$ and radius $r>0$.
We recall that $\mu_d$ denotes the Lebesgue measure in $\C^d$.
We have
\begin{multline*}
\norm{\left(1+\abs{z'}^2\right)\bar{z}_i z_j\abs{z''}^{2\beta-4}}^p_{L^p(B_R' \times B_{\epsilon}'')}
=\left(\int_{B_R'}\left(1+\abs{z'}^2\right)^p d\mu_m\right)
\\
\times \left(\int_{B_{\epsilon}''}(\abs{z_i} \abs{z_j})^p \abs{z''}^{(2\beta-4)p}d\mu_{n-m}\right).
\end{multline*}

The first integral in the right-hand side is clearly finite.
For the second one, using standard formula (see e.g. \cite[p. 23]{Kli91}) we have
$$\int_{B_{\epsilon}''}(\abs{z_i} \abs{z_j})^p \abs{z''}^{(2\beta-4)p}d\mu_{n-m}
=C(n-m,\epsilon,p)\int_0^{\epsilon} r^{2(n-m)-1+2p+(2\beta-4)p} dr.
$$
It is well-known that this integral is finite if, and only if, the power of $r$ is larger than $-1$, which gives the desired condition \eqref{cond w2p}.

\end{proof}

\begin{proof}[Proof of \ref{Pog ex calcul MAk}]
Let us now compute $\MA_k(\lambda(\DDC u))$.
We recall that
$$
\MA_k(\lambda)=
\prod_{(i_1,\ldots,i_k) \in E_n^k} \left(\lambda_{i_1}+\ldots+\lambda_{i_k}\right).
$$

From what precedes we know that there are $m$ eigenvalues of the form $\abs{z''}^{2\beta}$ and $n-m$ eigenvalues of the form $\abs{z''}^{2\beta-2}$ (by ``of the form'' we mean up to the multiplication by a smooth positive function).
Now, observe that:
\begin{itemize}
\item
The sum $\lambda_{i_1}+\ldots+\lambda_{i_k}$ is of the form $\abs{z''}^{2\beta}$ if so is each element of the sum.
There are $C_m^k$ possibilities for this situation to happen ($C_m^k=0$ if $m<k$).

\item
In all the other cases, the sum $\lambda_{i_1}+\ldots+\lambda_{i_k}$ is of the form $\abs{z''}^{2\beta-2}$.
Since there are $C_n^k=\card E_n^k$ products overall in $\MA_k$, this means that such sums appear $C_n^k-C_m^k$ times.
\end{itemize}
In summary, we have
$$\MA_k(\lambda(\DDC u))
=\psi(z) \abs{z''}^{C_m^k(2\beta)+(C_n^k-C_m^k)(2\beta-2)},$$
for some smooth positive function $\psi$.
Consequently, $\MA_k(\lambda(\DDC u)) \in C^{\infty}(\C^n)$ when the power of $\abs{z''}$ is exactly equal to zero, which gives the desired condition on $\beta$, namely
$$\beta=1-\frac{C_m^k}{C_n^k}.$$
For this value of $\beta$ we also have $\MA_k(\lambda(\DDC u))=\psi>0$ in $\C^n$.

\end{proof}

\section*{Acknowledgements}

Both authors were supported by the Polish National Science Centre Grant 2017/26/E/ST1/00955.
The authors would like to thank S\l awomir Dinew for numerous
discussions and fruitful suggestions.

\bibliographystyle{amsalpha}
\bibliography{biblio}

\end{document}